\documentclass[12pt,a4paper,reqno]{amsart}
\usepackage[utf8]{inputenc}
\usepackage{amsthm}
\usepackage{amssymb}
\usepackage[abbrev]{amsrefs}
\usepackage{bm}
\usepackage[shortlabels]{enumitem}
\usepackage{hyperref}
\usepackage{bbm}
\newtheorem{thm}{}[section]
\newtheorem{theorem}[thm]{Theorem}
\newtheorem{corollary}[thm]{Corollary}
\newtheorem{lemma}[thm]{Lemma}
\newtheorem{proposition}[thm]{Proposition}

\newtheorem*{claim}{Claim}
\theoremstyle{definition}

\theoremstyle{remark}
\newtheorem{remark}[thm]{Remark}
\newtheorem{question}[thm]{Question}

\numberwithin{equation}{section}
\allowdisplaybreaks
\AtBeginDocument{\def\MR#1{}}
\newcommand{\Lip}{\ensuremath{\mathrm{Lip}}}
\newcommand{\tqg}{\ensuremath{\bm{\Lambda}}}
\newcommand{\FF}{\ensuremath{\mathbb{F}}}
\newcommand{\RR}{\ensuremath{\mathbb{R}}}
\newcommand{\CC}{\ensuremath{\mathbb{C}}}
\newcommand{\NN}{\ensuremath{\mathbb{N}}}
\newcommand{\xx}{\ensuremath{\bm{x}}}
\newcommand{\zz}{\ensuremath{\bm{z}}}
\newcommand{\XX}{\ensuremath{\mathbb{X}}}
\newcommand{\XB}{\ensuremath{\mathcal{X}}}
\newcommand{\YY}{\ensuremath{\mathbb{Y}}}

\newcommand{\Ind}{\ensuremath{\mathbbm{1}}}

\newcommand{\EE}{\ensuremath{\mathbb{E}}}
\newcommand{\Cu}{\ensuremath{\mathcal{Q}}}
\newcommand{\Du}{\ensuremath{\mathcal{D}}}
\newcommand{\kql}{\ensuremath{\bm{\Gamma}}}
\newcommand{\GG}{\ensuremath{\mathcal{G}}}
\DeclareMathOperator{\sgn}{sign}

\DeclareMathOperator{\supp}{supp}
\newcommand{\abs}[1]{\left\lvert#1\right\rvert}
\newcommand{\norm}[1]{\left\lVert#1\right\rVert}
\newcommand{\floor}[1]{\left\lfloor #1 \right\rfloor}

\begin{document}
\title[]{Elton's near unconditionality of bases as a threshold-free form of greediness}
\thanks{F. Albiac acknowledges the support of the Spanish Ministry for Science and Innovation under Grant PID2019-107701GB-I00 for \emph{Operators, lattices, and structure of Banach spaces}. F. Albiac and J.~L. Ansorena acknowledge the support of the Spanish Ministry for Science, Innovation, and Universities under Grant PGC2018-095366-B-I00 for \emph{An\'alisis Vectorial, Multilineal y Aproximaci\'on.} M. Berasategui was supported by ANPCyT PICT-2018-04104.}
\author[Albiac]{Fernando Albiac}
\address{Department of Mathematics, Statistics, and Computer Sciencies--InaMat2\\
Universidad P\'ublica de Navarra\\
Campus de Arrosad\'{i}a\\
Pamplona\\
31006 Spain}
\email{fernando.albiac@unavarra.es}

\author[Ansorena]{Jos\'e L. Ansorena}
\address{Department of Mathematics and Computer Sciences\\
Universidad de La Rioja\\
Logro\~no\\
26004 Spain}
\email{joseluis.ansorena@unirioja.es}

\author[Berasategui]{Miguel Berasategui}
\address{Miguel Berasategui\\
IMAS - UBA - CONICET - Pab I, Facultad de Ciencias Exactas y Naturales\\
Universidad de Buenos Aires\\
(1428), Buenos Aires, Argentina}
\email{mberasategui@dm.uba.ar}
\begin{abstract}
Elton's near unconditionality and quasi-greediness for largest coefficients are two properties of bases that made their appearance in functional analysis from very different areas of research. One of our aims in this note is to show that, oddly enough, they are connected to the extent that they are equivalent notions. We take advantage of this new description of the former property to further the study of the threshold function associated with near unconditionality. Finally, we made a contribution to the isometric theory of greedy bases by characterizing those bases that are $1$-quasi-greedy for largest coefficients.
\end{abstract}
\maketitle
\section{Introduction}\noindent
Bases and basic sequences have played a decisive role in the development of modern Banach space theory. In particular, the possibility to extract a subsequence with some added features (such as being unconditional, symmetric, or weakly null) from a given kind of sequence in a Banach space has been, and continues to be, a technique of major usage throughout. The subsequence extraction principles were \emph{in vogue} in Banach space theory in the 1970's and culminated in the attainment of Rosenthal's $\ell_{1}$-theorem \cite{Rosenthal1974}, which states that a Banach space $\XX$ either contains $\ell_1$ or every bounded sequence in $\XX$ has a weakly Cauchy subsequence. In this ambience, the following problem arose: given a weakly null, normalized sequence in a Banach space, can we pass to a subsequence that is a basic sequence and is
\emph{in some sense} close to being unconditional? There are various ways in which one can make this vague question precise, and in many situations it has a positive answer.

The first mover in this direction was Elton, who in his Ph.D.~thesis \cite{Elton1978} proved that for $a\in(0,1]$ there exists a constant $C=C(a)<\infty$ such that every normalized weakly null sequence in a (real) Banach space admits a subsequence $(\xx_n)_{n=1}^{\infty}$ with the following property: if $\alpha_n\in [-1,1]$ for all $i\in \NN$ and $A\subset \{ n \in \NN\colon \abs{\alpha_n}\ge a\}$ then
\[
\norm{ \sum_{n\in A}\alpha_n \, \xx_n}\le C \norm{ \sum_{n=1}^{\infty}\alpha_n\, \xx_n}.
\]
To put this property in the context of our paper, we introduce some initial terminology. Let $\XX$ be a Banach (or more generally a quasi-Banach) space with a \emph{basis} $\XB=(\xx_n)_{n=1}^\infty$, i.e., $\XB$ is a norm-bounded sequence whose linear span $\langle \XB \rangle$ is dense in $\XX$ and for which there is a (unique) norm-bounded sequence of linear functionals $\XB^*=(\xx_n^*)_{n=1}^\infty$ in $\XX^*$ biorthogonal to $\XB$, that is $\xx_n^*\left(\xx_k\right)=\delta_{n,k}$ for all positive integers $n$ and $k$.

We consider the set $\Cu$ of vectors in $\XX$ whose coefficients (relative to $\XB$) belong to the unit ball of $\ell_{\infty}$, i.e.,
\[
\Cu=\Cu[\XB,\XX]=\{f\in \XX \colon \forall n\in \NN \, \abs{\xx_n^*(f)} \le 1\}.
\]
Now, given a number $a\ge 0$ and $f\in \XX$ put
\[
A(a,f):=\{n\in\NN \colon \abs{\xx_n^*(f)}\ge a\}.
\]
In this language, the basis $\XB$ is \emph{nearly unconditional} if for each $a\in(0,1]$ there is a constant $C=C(a)$ such that for all $f\in\Cu$, and all $A\subset A(a,f)$,\begin{equation}\label{nearlyuncdef}
\norm{ S_A(f)}\le C,
\end{equation}
where $S_A\colon\XX\to\XX$ is the the linear projection onto $\langle \xx_n \colon n\in A\rangle$ given by
\[
S_A(f)=\sum_{n\in A}\xx_n^*\left(f\right)\xx_n, \quad f\in\XX.
\]
For $a\in(0,1]$ we define $\phi(a)$ as the smallest value of the constant $C>0$ for which \eqref{nearlyuncdef} holds.
If inequality \eqref{nearlyuncdef} holds only for $A =A(a,f)$, the basis is said to be \emph{thresholding-bounded}, in which case we denote by $\theta(a)$ the least constant $C$.

Notice that near unconditionality is a threshold unconditionality property and that $\XB$ is unconditional if and only if $\phi$ is bounded. Hence, in a certain sense, the threshold function $\phi$ gives a measure of the conditionality of the basis. In turn, the boundedness of the function $\theta$ characterizes \emph{quasi-greedy bases}, a well-known concept by now in greedy approximation theory that was introduced by Konyagin and Temlyakov in 1999 \cite{KoTe1999}. Recall that, while a basis is unconditional in and only if there is a constant $C$ such that \eqref{nearlyuncdef} holds for all $f\in\XX$ and all $A\subset \NN$, a basis is quasi-greedy if there is a constant $C$ such that \eqref{nearlyuncdef} holds for all $f\in\XX$ and all $A\in\GG(f)$, where
\[
\GG(f)=\{A\subset\NN \colon \abs{\xx_n^*(f)} \ge\abs{\xx_k^*(f)} \mbox{ for all } (n,k)\in A \times\NN\setminus A\},
\]
is the set of \emph{greedy sets} of $f$.

The notion of near unconditionality was first linked to the thresholding greedy algorithm by Dilworth, Kalton and Kutzarova in \cite{DKK2003}. In that article, the authors proved that a basis is nearly unconditional if and only if it is thresholding bounded. This equivalence is perhaps a surprising result: in a natural sense, thresholding bounded bases seem to be as close to being quasi-greedy as near unconditional bases are to being unconditional; however, quasi-greedy bases need not be unconditional and the first examples that illustrate this were already built in \cite{KoTe1999}.

Further links between near unconditionality and the thresholding greedy algorithm were discovered in \cites{DOSZ2009, AABBL2022}. On the one hand, the authors of \cite{DOSZ2009} found a deep connection between near unconditionality and a long standing open problem in greedy approximation (see \cite{DOSZ2009}*{Problems 2 and 5, Proposition 13} for details). On the other hand, the authors of \cite{AABBL2022} obtained a characterization of near unconditionality in terms of the uniform boundedness of some nonlinear operators associated with bases. To be able to state this characterization we need some more notation.

We use $\EE$ for the set of scalars of modulus one. For $f$ in $\XX$, we will denote by $\varepsilon(f)\in\EE^\NN$ the sequence
\[
\varepsilon(f)=\left(\sgn\left(\xx_n^*\left(f\right)\right)\right)_{n=1}^\infty.
\]
Given $A\subset\NN$ finite and $\varepsilon\in\EE^{A}$ we put
\[
\Ind_{\varepsilon,A}=\sum_{n\in A}\varepsilon_n\, \xx_n.
\]
It is known (see \cite{AABW2021}*{Lemma 4.12}) that if $(\xx_n)_{n=1}^{\infty}$ is a quasi-greedy basis of a quasi-Banach space $\mathbb X$ then there is constant $K$ such that
\begin{equation}\label{TQG}
\min_{n\in A}\abs{ \xx_n^*(f)} \norm{ \Ind_{\varepsilon(f),A}} \le K\norm{ f}, \quad f\in\XX, \, A\in\GG(f).
\end{equation}

The importance of the estimate \eqref{TQG} in the study of greedy-like bases was already implicit in the work of Dilworth at al.\ (see \cite{DKKT2003}*{Lemma 2.2}), but it was not until recently (see \cite{AABBL2022}) that these bases were singled out and given the name of
\emph{truncation quasi-greedy}. This brand new greedy-like property also has its own thresholding counterpart which turns out to be equivalent to near unconditionality.

A basis $\XB=(\xx_n)_{n=1}^{\infty}$ is said to be \emph{nearly truncation quasi-greedy} (see \cite{AABBL2022}*{Definition 3.1}) if for every $0<a\le 1$ there is $K=K(a)>0$ such that
\begin{equation}\label{NTQG}
\min_{n\in A(f,a)}\abs{ \xx_n^*(f)} \norm{ \Ind_{\varepsilon(f),A(f,a)}} \le K\norm{ f}, \quad f\in\Cu.
\end{equation}
Given $a\in(0,1]$, we will denote by $\lambda(a)$ the smallest value of $K$ for which \eqref{NTQG} holds. The function $\lambda$ is bounded if and only $\XB$ is truncation quasi-greedy, and, if we denote by $\tqg$ the optimal constant $K$ such that \eqref{TQG} holds, then
\[
\tqg=\sup_{0<a\le 1} \lambda(a).
\]
Despite the fact that nearly truncation quasi-greedy bases are thresholding bounded, in the literature we find examples of truncation quasi-greedy bases that are not quasi-greedy (see, e.g., \cite{DKK2003}*{Example 4.8} and \cite{BBG2017}*{Proposition 5.6}).

This note is motivated by the attempt to find a characterization of nearly unconditional bases which, unlike the two already existing ones, does not depend on a threshold function. We accomplish that by showing that nearly unconditional bases admit a simple characterization in terms of yet another property that arises from the study of the thresholding greedy algorithm, namely, quasi-greediness for largest coefficients.

We recall that a basis is \emph{quasi-greedy for largest coefficients} (QGLC for short) if there is a constant $L$ such that
\begin{equation}\label{QGLCdef}
\norm{\Ind_{\varepsilon,A}}\le L \norm{\Ind_{\varepsilon,A}+f}
\end{equation}
for all finite sets $A\subset \NN$, all $\varepsilon\in\EE^A$, and all $f\in\Cu$ such that $\supp(f)\cap A=\emptyset$ (see \cite{AABW2021}*{Definition 4.6}). If the above holds for a given $L\in[1,\infty)$, we say that the basis is $L$-QGLC, and the optimal constant $L$ will be denoted by $\kql$. Clearly, $\kql\le\tqg$. Hence, QGLC bases are truncation quasi-greedy.

In Section~\ref{section: main} we tackle the aforementioned characterization of nearly unconditional bases.
We also study the growth in terms of $\kql$ of the threshold functions associated with unconditionality, quasi-greediness, and truncation quasi-greediness. For that, it will be convenient to consider the following variation of the truncation quasi-greedy threshold function $\lambda$.

Given a basis $\XB$ be a basis of a quasi-Banach space $\XX$ and $0<a\le 1$, we denote by $\rho$ the smallest constant $K$ such that
\[
a \norm{\Ind_{\varepsilon(f),A(f,a)}} \le K\norm{ f}, \quad f\in\Cu.
\]
Notice that $\rho(a) \le \lambda(a)$ for all $a\in(0,1]$. Morever, since
\[
A(f,a)=A\left(f,\min_{n\in A(f,a)}\abs{ \xx_n^*(f)}\right), \quad f\in\Cu, \, 0<a\le 1,
\]
we have
\begin{equation*}
\sup_{0<a\le 1} \rho(a)=\sup_{0<a\le 1} \lambda(a).
\end{equation*}
Hence, a basis is truncation quasi-greedy if and only if the function $\rho\colon(0,1]\to[1,\infty]$ is bounded.

In Section~\ref{section: further}, we prove further results on nearly unconditional bases, and we study the growth as $a$ goes to zero of the numbers $\phi(a)$, $\theta(a)$, $\lambda(a)$ and $\rho(a)$. The investigation carried out in Section~\ref{sect:iso} lies within the topic of studying greedy-like basis from an isometric point of view. In it, we characterize $1$-QGLC bases. We close with some questions that arise naturally from our work and that we gathered in Section~\ref{sect:questions}.

Throughout this paper, we will use standard quasi-Banach space and greedy approximation terminology as can be found in \cite{AABW2021}. For the reader's ease, let simply point out that, by the Aoki-Rolewicz theorem, any quasi-Banach space is locally $p$-convex for some $0<p\le 1$, hence it is a $p$-Banach space under a renorming. Consequently, any quasi-Banach space $\XX$ can be equipped with an equivalent quasi-norm $\norm{\,\cdot\,} \colon\XX\to[0,\infty)$ which is a continuous map. All quasi-Banach spaces below are assumed to be endowed with such a quasi-norm.

\section{Characterization of nearly unconditional bases}\label{section: main}\noindent
Throughout this paper, we will adopt the convention that the threshold numbers $\phi(a)$, $\theta(a)$, $\lambda(a)$ and $\rho(a)$, $0<a\le 1$, associated with the notions of unconditionality, quasi-greediness and truncation quasi-greediness, as well as the number $\kql$ linked with quasi-greediness for largest coefficients, are defined on general bases, so they may take a priori the value infinity.

It is clear that $\theta\le\phi$ and that the function $\phi$ is non-increasing. It is known that the threshold functions $\theta$ and $\lambda$ are non-increasing as well (see \cite{DKK2003}*{Proposition 4.1}) and \cite{AABBL2022}*{Lemma 3.2}, respectively) and so is $\rho$. Indeed, given $0<a\le b \le 1$ and $f\in\Cu$, the function $g:=a^{-1}b f$ belongs to $\Cu$, and so
\[
a\norm{\Ind_{\varepsilon(f),A(f,a)}}=\frac{a}{b} b\norm{ \Ind_{\varepsilon(g),A(g,b)}} \le \frac{a}{b} \rho(b)\norm{g}=\rho(b)\norm{f}.
\]

We start our study with a lemma that relates the threshold functions at level $1$.
\begin{lemma}\label{remark: similar}
Let $\XB$ be a basis of a quasi-Banach space. Then,
\[
\phi(1)=\theta(1)=\lambda(1)=\rho(1)=\kql.
\]
\end{lemma}
\begin{proof}
From the definitions it follows immediately that $\rho(1)=\lambda(1)=\theta(1)$, and $\phi(1)=\kql$. To prove that $\phi(1)\le \theta(1)$, we use a perturbation technique. Let $f\in\Cu$ and $A\subset A(f,1)$. For each $\epsilon>0$ there is $f_{\epsilon}\in\Cu$ with $S_A(f)=S_A(f_\epsilon)$, $A=A(f,1)$ and $\norm{f-f_\epsilon}<\epsilon$. Since
\[
\norm{ S_A(f)}=\norm{ S_A(f_{\epsilon})} \le \theta(1)\norm{ f_{\epsilon}},
\]
letting $\epsilon$ tend to zero, we obtain the desired inequality.
\end{proof}

The submultiplicative behaviour of the threshold functions $\phi$, $\theta$ and $\rho$, which is made explicit in Lemma~\ref{lemma: morebounds} below, will be essential in this paper.
\begin{lemma}\label{lemma: morebounds}
Let $\XB$ be a nearly unconditional basis of a $p$-Banach space $\XX$, $0<p\le 1$. For any $0<a, b\le 1$ we have
\begin{align}
\phi(ab)&\le \left((1-b)^p\phi^{\,p}(b)+ \phi^{\, p}(a) \left( 1+ (1-b)^p \theta^{\, p} (b) \right) \right)^{1/p} \label{phimorebounds}\\
\theta(ab)&\le
\left((1-b)^p\theta^{\, p}(b)+\theta^{\, p}(a) \left( 1+ (1-b)^p \theta^{\, p} (b) \right) \right)^{1/p}\label{thetamorebounds}\mbox{, and}\\
\rho\left(ab\right)&\le \rho(a)\left(1+\left((1-b)\theta(b)\right)^p\right)^{1/p}.\label{rhomorebounds}
\end{align}
\end{lemma}
\begin{proof}
Given $f\in\Cu$ we set
\[
g:=\frac{1}{b}\left(f-(1-b)S_{A(f,b)}(f)\right)=\frac{1}{b} \left( f-S_{A(f,b)}\right)+ S_{A(f,b)}(f).
\]
Then,
\begin{equation}\label{eq:anso1}
b^p \norm{ g}^p\le
\norm{f}^p+(1-b)^p \norm{S_{A(f,b)}(f)}^p
\le ( 1+ (1-b)^p \theta^{\, p} (b) ) \norm{f}^p.
\end{equation}
Given $A\subset A(f,ab)$, we set $B:=A\cap A(f,b)$. Since $b S_A(g)=S_A(f) +(1-b)S_B(f)$,
\begin{equation}\label{eq:anso2}
\norm{S_A(f)}^p \le b^p \norm{S_A(g)}^p + (1-b)^p \norm{S_B(f)}^p.
\end{equation}
In turn, since $g\in\Cu$, $A\subset A(g,a)$, and $B\subset A(f,b)$,
\begin{equation}\label{eq:anso3}
\norm{S_A(g)} \le \phi(a) \norm{g} \quad \mbox{ and }\quad \norm{S_B(f)} \le \phi(b) \norm{f}.
\end{equation}
Moreover, in the particular case that $A=A(f,ab)$, we have $A=A(g,a)$ and $B=A(f,b)$, and so
\begin{equation}\label{eq:anso4}
\norm{S_A(g)} \le \theta(a) \norm{g} \quad \mbox{ and }\quad \norm{S_B(f)} \le \theta(b) \norm{f}.
\end{equation}
Finally, since $\varepsilon(g)=\varepsilon(f)$,
\begin{equation}\label{eq:anso5}
ab\norm{ \Ind_{\varepsilon(f),A(f,ab)}} \le b \rho(a) \norm{ g}.
\end{equation}

Combining \eqref{eq:anso1}, \eqref{eq:anso2} and \eqref{eq:anso3} (resp., \eqref{eq:anso4}) gives \eqref{phimorebounds} (resp., \eqref{thetamorebounds}). In turn, combining \eqref{eq:anso1} with \eqref{eq:anso5} gives \eqref{rhomorebounds}.
\end{proof}

We will use the following elementary lemma a couple of times.

\begin{lemma}\label{lem:anso9}
Suppose $f\colon(0,1]\to[0,\infty)$ is a non-increasing function such that for some $0<a<1$, $C\in[0,\infty)$ and $D\in(1,\infty)$,
\[
f(a^n)\le C+Df(a^{n-1}), \quad n\in\NN.
\]
Then,
\[
f(t)+\frac{C}{D-1} \le D\left( f(1)+\frac{C}{D-1}\right) t^{\log_a D}, \quad 0<t\le 1.
\]
\end{lemma}
\begin{proof}
Replacing $f$ with $f+C/(D-1)$, we can assume that $C=0$. By induction, $f(a^n)\le f(1) D^n $ for all $n\in\NN\cup\{0\}$. Given $0<t\le 1$, pick $n\in\NN\cup\{0\}$ such that $a^n<t\le a^{n-1}$. We have
\[
f(t)\le f(a^n) \le f(1) D^n=f(1) D \left( a^{n-1}\right)^{\log_a D}\le D f(1) t^{\log_a D}.\qedhere
\]
\end{proof}

Given $0<a<1$, we say that a basis $\XB$ of a quasi-Banach space $\XX$ is nearly unconditional at level $a$ if $\phi(a)<\infty$, where, as usual, $\phi$ denotes the unconditionality threshold function of the basis. For locally convex spaces, i.e., Banach spaces, combining \cite{DKK2003}*{Propositions 4.1 and 4.5} yields that a basis is nearly unconditional at level $a$ for some $0<a<1$ if and only if it is nearly unconditional. It must be conceded that the proof given by the authors of \cite{DKK2003} can be adapted to the more general setting of quasi-Banach spaces. Still, for the sake of completeness and clarity, we write down a proof of this result that takes into account the specific traits of nonlocally convex spaces.

\begin{lemma}\label{lem:fromsomeatoanya}
Let $\XB$ be a basis of a $p$-Banach space $\XX$, $0<p\le 1$. Suppose that $\XB$ is nearly unconditional at some level $c\in(0,1)$. Then $\XB$ is nearly unconditional. Moreover, there are $C\in[1,\infty)$ and $d\in(0,\infty)$ only depending on $\phi(c)$ and $p$ such that the unconditionality threshold function $\phi$ satisfies
\[
\phi(a) \le C a^{-d}, \quad 0<a\le 1.
\]
In fact, the inequality holds with
\begin{align*}
C&=(1+(1-c)^p\phi^{\, p}(c))^{1/p}(\phi^{\,p}(1)+1)^{1/p} \quad\mbox{ and} \\
d&=-\frac{1}{p} \log_c \left(1+(1-c)^p\phi^{\, p}(c)\right) .
\end{align*}
\end{lemma}

\begin{proof}
Applying inequality~\eqref{phimorebounds} with $a=c^{n-1}$ and $b=c$ yields
\[
\phi^{\, p}(c^n) \le (1-c)^p\phi^{\, p}(c) +\left(1+(1-c)^p\phi^{\, p}(c)\right) \phi^{\, p}(c^{n-1}), \quad n\in\NN.
\]
By Lemma~\ref{lem:anso9}, for all $a\in(0,1]$ we have
\[
\phi(a) \le \left( 1+\phi^{\, p}(a) \right)^{1/p} \le \left( C^p a^{-dp}\right)^{1/p}=C a^{-d}.\qedhere
\]
\end{proof}

We are now in a position to establish the equivalence between quasi-greediness for largest coefficients and near unconditionality. Prior to that, we bring up a result from \cite{AABW2021} that was necessary to show that quasi-greedy bases in quasi-Banach spaces are truncation quasi-greedy. Recall that a basis is \emph{$C$-suppression unconditional for constant coefficients} ($C$-SUCC for short) if
\[
\norm{\Ind_{\varepsilon,B}}\le C\norm{\Ind_{\varepsilon,A}}
\]
for all finite sets $B\subset A\subset \NN$ and all $\varepsilon\in\EE^A$. Note that if $\XB$ is $C$-QGLC, it is also $C$-SUCC (cf.\ \cite{AABW2021}*{Lemma 4.7}).

\begin{lemma}\cite{AABW2021}*{Lemmas 3.2, 3.6 and 4.7}\label{lemmaSUCC}
Let $\XB$ be a basis of a $p$-Banach space $\XX$, $0<p\le 1$. If $\XB$ is $C$-SUCC there are positive positive constants $C_1\ge 1$, $s>1$ depending only on $C$ and $p$ such that
\[
\norm{ \sum_{n\in A}a_n\xx_n}\le C_1\norm{\sum_{n\in A}b_n\xx_n}
\]
for every finite set $A\subset \NN$ and scalars $(a_n)_{n\in A},(b_n)_{n\in A}$ with the property that, for all $n,j\in A$,
\[
\abs{a_n}\le 1 \le\abs{b_j}\le s.
\]
\end{lemma}

\begin{theorem}\label{theorem: QGLC=NU}
Let $\XB$ be a basis of a $p$-Banach space $\XX$. If $\XB$ is QGLC and $\XX$ is a $p$-Banach space, then $\XB$ is nearly unconditional,and the unconditionality threshold function $\phi$ satisfies
\[
\phi(a) \le C a^{-d}, \quad 0<a\le 1,
\]
for some constants $C\in[1,\infty)$ and $d\in(0,\infty)$ only depending on $\kql$ and $p$. Conversely, if $\XB$ is nearly unconditional, there is $L$ depending only on $\phi$ such that $\XB$ is $L$-QGLC.
\end{theorem}

\begin{proof}
Nearly unconditional bases are quasi-greedy for largest coefficients by Lemma~\ref{remark: similar}. To prove the converse, we consider the constants $C_1\in[1,\infty)$ and $s\in(1,\infty)$ provided by Lemma~\ref{lemmaSUCC}. Choose $a\in(0,1)$ close enough to $1$ so that
\[
\frac{1}{a} <s \quad \mbox{ and } \quad C_1^{2p}\kql^p \frac{(1-a)^p}{a^{p}}\le \frac{1}{2}.
\]
Given $f\in \Cu$ and $A\subset A(f,a)$, we have
\begin{align*}
\norm{S_A(f)}^p
&\le C_1^p\norm{\Ind_{\varepsilon(f),A}}^p\\
&\le C_1^p \kql^p\norm{ \Ind_{\varepsilon(f),A} + f- S_A(x)}^p\\
&\le C_1^p\kql^p\norm{f}^p+C_1^p\kql^p \norm{ \sum_{n\in A}\sgn(\xx_n^*(f) \left( 1 -\abs{\xx_n^*(f)}\right)\xx_n}^p\\
&\le C_1^p\kql^p\norm{f}^p+C_1^{2p}\kql^p\norm{ \sum_{n\in A}a^{-1}(1-a)\xx_n^*(f)\xx_n}^p\\
&=C_1^p\kql^p\norm{f}^p+C_1^{2p}\kql^p \frac{(1-a)^p}{ a^{p}} \norm{S_A(f)}^p.
\end{align*}
Thus,
\[
\norm{S_A(f)} \le 2^{1/p} C_1\kql\norm{f}.
\]
It follows that $\XB$ is nearly unconditional at level $a$, with $\phi(a) \le 2^{1/p}C_1\kql$. By Lemma~\ref{lem:fromsomeatoanya}, we are done.
\end{proof}

\begin{remark}\label{rem:scale}
In light of Theorem~\ref{theorem: QGLC=NU}, we can give a more precise formulation of \cite{AABBL2022}*{Theorem 3.4} involving the function $\rho$. To be precise, since the proof of \cite{AABBL2022}*{Lemma 3.3} works replacing $\lambda$ with $\rho$, if $\XB$ is a QGLC basis of a $p$-Banach space, $0<p\le 1$, then there are constants $C_1$ and $C_2$ depending only of $\kql$ and $p$ such that
\[
\rho(a) \le \lambda(a)\le C_1\theta(a)\quad \mbox{ and } \quad
\theta(a) \le \phi(a) \le C_2 \frac{\rho(a)}{a}, \quad 0<a\le 1.
\]
It is also worth mentioning that
\[
\lambda(a)\le C_3 \phi(a), \quad 0<a\le 1,
\]
where the constant $C_3$ only depends on $p$. Indeed, given $0<a\le 1$, and $f\in\Cu$, by \cite{AABW2021}*{Corollary 2.3},
\[
\norm{\sum_{n\in A(f,a)} a_n \, \sgn(\xx_n^*(f)) \, \xx_n} \le A_p \phi(a) \norm{f}
\]
whenever $0\le a_n \le 1$, where $A_p$ is the geometric constant defined by
\begin{equation}\label{eq:Ap}
A_p= (2^p-1)^{1/p}.
\end{equation}
Choosing $a_n =b/\abs{\xx_n^*(f)}$, where
\[
b=\min_{n\in A(f,a)} \abs{\xx_n^*(f))},
\]
we obtain the desired inequality with $C_3=A_p$.
\end{remark}

If the space is locally convex we can give precise estimates in terms of $\kql$ for the constants $C$ and $d$ in Theorem~\ref{theorem: QGLC=NU}. Our approach to these estimates relies on proving that the function $\phi$ is ``smooth enough''.

Given a real interval $I$ and $0<p\le 1$, a function $\psi \colon I \subset \RR \to \FF$ is said to be \emph{$p$-Lipschitz} if
\[
\Lip_p(\psi)=\Lip_p(\psi,I):=\sup_{\substack{s,t\in I \\ s\not= t}}\frac{ \abs{\psi(t)-\psi(s)}}{\abs{t-s}^p}<\infty.
\]

\begin{proposition}\label{prop:AnsoLip}
Let $\XB$ be a nearly unconditional basis of a $p$-Banach space $\XX$, $0<p\le 1$. Then $\phi$, $\theta$ and $\rho$ are $p$-Lipschitz on $[c,d]$ for every $0<c<d\le 1$. Moreover, their $p$-Lipschitz constants satisfy
\begin{align*}
\Lip_p(\phi,[c,d])&\le \frac{\phi^{\, 1-p}(c) \phi^{\, p}(c/d) +\phi(c) \theta^{\, p}(c/d)}{p c^p},\\
\Lip_p(\theta,[c,d]) &\le \frac{ \theta^{\, p}(c/d) (\theta^{\, 1-p}(c)+\theta(c))}{p c^p} \mbox{, and}\\
\Lip_p(\rho,[c,d])&\le \frac{\rho(c) \theta^{\, p}(c/d)}{p c^p}.
\end{align*}
In particular, $\phi$, $\theta$ and $\rho$ are continuous on $(0,1]$.
\end{proposition}

\begin{proof}
If $a\in[c,d]$ and $b\in(0,1)$ are such that $ab\ge c$, then $b\ge c/d$. Taking into account that $\phi$, $\theta$ and $\rho$ are non-increasing, an application of Lemma~\ref{lemma: morebounds} gives
\begin{align*}
\frac{\phi^{\, p}(ab)- \phi^{\, p}(a)}{(a -ab)^{p}} &\le \frac{\phi^{\,p}(c/d)+\phi^{\,p}(c)\theta^{\, p}(c/d)}{c^p},\\
\frac{\theta^{\, p}(ab)- \theta^{\, p}(a)}{(a -ab)^{p}}&\le \frac{\theta^{\, p}(c/d)(1+\theta^{\, p}(c))}{c^p} \mbox{, and}\\
\frac{\rho^{\, p}(ab)- \rho^{\, p}(a)}{(a -ab)^{p}}&\le \frac{\rho^{\, p}(c)\theta^{\, p}(c/d)}{c^p},
\end{align*}
Combining these inequalities with the elementary estimate
\[
\frac{x-y}{x^p -y^p} \le \frac{x^{1-p}}{p}, \quad 0<y<x,
\]
yields the desired estimates for the Lipschitz constants.
\end{proof}

The Euclidean distance $\abs{\,\cdot\,}$ on $\RR$ satisfies the inequality
\begin{equation}\label{eq:metric}
\abs{x_N-x_0} \le \sum_{j=1}^n \abs{x_j-x_{j-1}},\quad\; x_0<\cdots<x_j<\cdots <x_N.
\end{equation}
However, this estimate does not hold when we replace $\abs{\,\cdot\,}$ with any of its snowflakings $\abs{\,\cdot\,}^p$, $0<p<1$, even after adding a multiplicative constant. This obstruction compels us to establish the following consequence of Proposition~\ref{prop:AnsoLip} only for locally convex spaces.

\begin{corollary}\label{proposition: continuity}
Let $\XB$ be a nearly unconditional basis of a Banach space $\XX$. Then, $\phi$, $\theta$ and $\rho$ are Lipschitz on $[c,1]$ for each $0<c<1$, with respective Lipschitz constants bounded as follows:
\begin{align*}
\Lip(\phi,[c,1])&\le \frac{\kql (1+\phi(c))}{c},\\
\Lip(\theta,[c,1])&\le \frac{\kql (1+\theta(c))}{c},\\
\Lip(\rho,[c,1])&\le \frac{\kql\rho(c)}{c}.
\end{align*}
\end{corollary}
\begin{proof}
Let $\psi$ be one the the functions $\phi$, $\theta$ or $\rho$, and set
\[
L_\phi=\frac{1+\phi(c)}{c}, \quad
L_\theta= \frac{ 1+\theta(c)}{c},\quad
L_\rho= \frac{\rho(c)}{c}.
\]

Given $\epsilon>0$, the continuity of $\phi$ at $1$ combined with Lemma~\ref{remark: similar} gives $0<\delta<1$ such that $\theta(\delta)\le\phi(\delta) \le \kql+\epsilon$. Let $(c_j)_{j=1}^N$ be a partition of $[c,1]$ with $c_{j-1}/c_j \ge \delta$ for $j=1$, \dots, $N$. By Proposition~\ref{prop:AnsoLip},
\[
\Lip(\psi,[c_{j-1},c_j])\le (\kql+\epsilon) L_\psi, \quad j=1,\dots, N.
\]
Taking into account \eqref{eq:metric}, we infer that $\Lip(\psi,[c,1]) \le (\kql+\epsilon) L_\psi$. Since $\epsilon>0$ is arbitrary, we are done.
\end{proof}

\begin{theorem}\label{corollary: QGLCphibound}
Let $\XB$ be a basis of a Banach space $\XX$. If $\XB$ is quasi-greedy for largest coefficients then the functions
\[
a\mapsto (\phi(a)+1)a^{\kql}, \quad
a\mapsto (\theta(a)+1)a^{\kql}, \quad
a\mapsto \rho(a)a^{\kql},
\]
are non-decreasing on $(0,1]$. Thus, in particular,
\[
\phi(a)\le \frac{\kql+1}{a^{\kql}}-1, \quad \rho(a)\le \frac{\kql}{a^{\kql}}, \quad 0<a\le 1.
\]
\end{theorem}

\begin{proof}
Set $b_\phi=b_\theta=1$, and $b_\rho=0$, and let $\psi$ be one of the functions $b_\phi+\phi$, $b_\theta+\theta$ or $b_\rho+\rho$. Since Lipschitz functions are absolutely continuous, from Corollary~\ref{proposition: continuity} we infer that $\psi$ is locally absolutely continuous, i.e., absolutely continuous on each closed subinterval of $(0,1]$, and
\[
-\psi'(a)=\abs{\psi'(a)} \le \frac{\kql \psi(a)}{a}, \quad \mbox{ a.e. } a\in(0,1].
\]
Hence, the function $\tau\colon (0,1]\to \RR$ given by
\[
\tau(a)=\log(\psi(a))+\kql \log (a)
\]
is locally absolutely continuous, and $\tau'\ge 0$ almost everywhere. Hence, $\tau$ is non-decreasing. Therefore, $e^\tau$ is non-decreasing, as desired. In particular, by Lemma~\ref{remark: similar},
\[
\psi(a) a^{\kql} \le \psi(1) =b_\psi+\kql.\qedhere
\]
\end{proof}

We put an end to this section with a characterization of quasi-greediness for largest coefficients in terms of a formally weaker property, which might simplify the computations required to determine that a basis is nearly unconditional. To prove it, it will be convenient to use a geometric constant introduced in \cite{AABW2021}. Given $0<p\le 1$ we set
\[
B_p=\begin{cases} 2^{1/p} A_p& \mbox{ if }\FF=\RR,\\ 4^{1/p} A_p& \mbox{ if }\FF=\CC,\end{cases}
\]
where $A_p$ is the constant defined in \eqref{eq:Ap}.

\begin{proposition}\label{proposition: onlyonesign}Let $\XB=(\xx_n)_{n=1}^\infty$ be a basis of a quasi-Banach space $\XX$. Suppose that there are $K>0$ and $\widetilde{\varepsilon}=(\widetilde{\varepsilon_n})_{n=1}^\infty\in \EE^{\NN}$ such that
\begin{align*}
\norm{ \Ind_{\widetilde{\varepsilon}, A}}\le K \norm{ \Ind_{\widetilde{\varepsilon}, A}+f }
\end{align*}
for all finite sets $A\subset \NN$ and all $f\in\Cu$ with $\supp(f)\cap A=\emptyset$. Then, $\XB$ is QGLC.
\end{proposition}

\begin{proof}
Although the proof can be simplified in the case when $\FF=\RR$, we will write down a unified proof that works for both real and complex spaces. Assume without lost of generality that $\XX$ is a $p$-Banach space, $0<p\le 1$. Replacing $\xx_n$ with $\widetilde{\varepsilon_n} \, \xx_n$ for each $n\in\NN$, we may assume that $\widetilde{\varepsilon}_n=1$ for all $n\in\NN$. Let then denote
\[
\Ind_A=\Ind_{\widetilde{\varepsilon},A}, \quad A\subset \NN, \, \abs{A}<\infty.
\]
We have $\norm{\Ind_B}\le K \norm{\Ind_A}$ whenever $B\subset A$. By $p$-convexity (see \cite{AABW2021}*{Corollary 2.4}),
\begin{equation}\label{eq:AnsoSUCC}
\norm{\sum_{n\in A} a_n \, \xx_n} \le B_p K \norm{\Ind_A}, \quad A\subset \NN, \, \abs{A}<\infty,\, \abs{a_n}\le 1.
\end{equation}
Set
\[
\delta= \left(1-2^{-p} \right)^{-1/p} B_p^{-1} K^{-2}=2A_p^{-1} B_p^{-1} K^{-2}.
\]
Fix $A\subset \NN$ finite, $f\in\Cu$ with $\supp(f)\cap A=\emptyset$, $\omega\in\EE$, and $\varepsilon=(\varepsilon_n)_{n\in A}$ with $\abs{\varepsilon_n-\omega}\le \delta$ for all $n\in A$. We have
\begin{align*}
\norm{\Ind_A}^p &\le K^p \norm{ \Ind_A+\omega^{-1} f}^p\\
&=K^p \norm{ \omega \Ind_A+ f}^p \\
& \le K^p\norm{\omega\Ind_A-\Ind_{\varepsilon,A}}^p+K^p\norm{\Ind_{\varepsilon,A} +f }^p\\
&\le B_p^p K^{2p} \delta^p \norm{\Ind_A}^p+K^p\norm{\Ind_{\varepsilon,A} +f }^p\\
&= \left( 1-2^{-p} \right) \norm{\Ind_A}^p +K^p\norm{\Ind_{\varepsilon,A} +f }^p
\end{align*}
Summing up,
\[
\norm{\Ind_A} \le 2K \norm{\Ind_{\epsilon,A} +f }.
\]

To obtain a similar estimate without assuming that the scalars in $\varepsilon$ are close enough to a suitable scalar $\omega\in\EE$, borrowing an idea from \cite{BB2022b} we pick a finite partition $(\EE_k)_{k=1}^N$ of $\EE$ for which there is $(\omega_k)_{k=1}^N\in\EE^N$ such that
\[
\sup_{\omega\in\EE_k} \abs{\omega-\omega_k}\le \delta, \quad k=1,\dots, N.
\]
Fix $B\subset A\subset \NN$ with $A$ finite, $\varepsilon=(\varepsilon_n)_{n\in A}\in\EE^A$, and $f\in\Cu$ with $\supp(f)\cap A=\emptyset$. Set
\[
B_k=\{n\in B \colon \varepsilon_n=\EE_k\}, \quad \mbox{ and } \quad f_k =\Ind_{\varepsilon,A\setminus B_k} +f, \quad k=1, \dots, N.
\]
Since $f_k\in\Cu$ and $\supp(f_k)\cap B_k=\emptyset$ for all $k=1$, \dots, $N$,
\[
\norm{\Ind_B}^p\le \sum_{k=1}^N \norm{\Ind_{B_k}}^p \le (2K)^p \sum_{k=1}^N \norm{\Ind_{\varepsilon,B_k} +f_k}^p=N (2K)^p \norm{\Ind_{\varepsilon,A} +f }^p.
\]
Set $L= 2 B_p N^{1/p} K$. Applying again \cite{AABW2021}*{Corollary 2.4} we obtain
\[
\norm{\sum_{n\in A} a_n \, \xx_n} \le L \norm{\Ind_{\varepsilon,A}+f }, \quad A\subset \NN, \, \abs{A}<\infty, \,\varepsilon\in\EE^A, \, \abs{a_n}\le 1,
\]
That is, $\XB$ is $L$-QGLC.
\end{proof}

\begin{remark}
The proof of Theorem~\ref{proposition: onlyonesign} works verbatim using
\[
\left\{ \Ind_{\varepsilon,B} \colon B\subset\NN,\; \abs{B}<\infty, \; \varepsilon\in\EE^B\right\}
\]
instead of $\Cu$. This way, we obtain a new characterization of suppression unconditionality for constant coefficients. Namely, a basis $\XB$ of a quasi-Banach space $\XX$ is SUCC if and only if there are $\widetilde{\varepsilon}\in\EE^{\NN}$ and $K>0$ such that
\[
\norm{ \Ind_{\widetilde{\varepsilon}, A}}\le K \norm{ \Ind_{\widetilde{\varepsilon}, A}+\Ind_{\varepsilon,B} }
\]
for all $A$, $B \subset \NN$ finite with $A\cap B=\emptyset$, and all $\varepsilon\in\EE^B$.
\end{remark}

\section{On the growth of the threshold functions}\label{section: further}\noindent
We get started by recalling a result that shows that the unconditionality threshold function of a truncation quasi-greedy basis satisfies a better estimate than the one provided by Theorem~\ref{theorem: QGLC=NU}.

\begin{theorem}[\cite{AAB2022}*{Theorem 6.5}]\label{thm:miguelbound}
Let $\XB$ be a truncation quasi-greedy basis of a $p$-Banach space $\XX$, $0<p\le 1$. Then there is a constant $C$ such that
\begin{equation*}
\phi(a)\le C ( 1-\log a)^{1/p},\quad 0<a\le 1.
\end{equation*}
\end{theorem}
Proposition~\ref{prop:anso8} below improves Theorem~\ref{thm:miguelbound}. In order to prove it we need an auxiliary lemma that allows us to estimate $\phi$ in terms of $\rho$.

\begin{lemma}\label{lem:logbounds,sumc}
Let $\XB$ be a basis of a $p$-Banach space $\XX$, $0<p\le 1$. Suppose $\XB$ is $C$-SUCC, $1\le C<\infty$. Then, there is a constant $C_1$ depending only on $C$ and $p$ such that, for all $n\in\NN$ and $0<a\le 1$,
\[
\phi(a^n)\le \frac{C_1}{a}\left(\sum_{k=1}^n \rho^{\, p} (a^{k})\right)^{1/p}.
\]
\end{lemma}

\begin{proof}
If $a=1$, the result follows from Lemma~\ref{remark: similar}. To show the result in the case when $0<a<1$ we note that
by \cite{AABW2021}*{Corollary 2.3},
\[
\norm{\sum_{n\in A} a_n\, \xx_n} \le A_p C \norm{ \Ind_{\varepsilon,A}}
\]
for all $A\subset\NN$ finite, $\varepsilon=(\varepsilon_n)_{n\in A}\in\EE^A$, and $(a_n)_{n\in A}\in\FF^A$ with $\abs{a_n} \le 1$ and $\sgn(a_n)=\varepsilon_n$. (Note that we can also derive such an estimate from Lemma~\ref{lemmaSUCC}.)

Pick $f\in\Cu$ and $A\subset A(f,a^n)$. Consider the partition $(A_k)_{k=1}^n$ of $A$ given by $A_1:=A\cap A(f,a)$ and
\[
A_k:=A\cap \left(A(f,a^{k})\setminus A(f,a^{k-1})\right),\quad k=2,\dots, n.
\]
By $p$-convexity, $\norm{S_A(f)}^p \le\sum_{k=1}^n \norm{S_{A_k}(f)}^p$. Then,
\[
\norm{S_{A_k}(f)}
=\frac{\norm{ S_{A_k}\left(a^{1-k}f\right) } }{a^{1-k} }
\le \frac{A_p C}{a^{1-k}}\norm{ \Ind_{\varepsilon(f),A(f,a^k)}}
\le \frac{A_p C}{a} \rho(a^k).\qedhere
\]
\end{proof}

\begin{proposition}\label{prop:anso8}
Let $\XB$ be a basis of a $p$-Banach space $\XX$, $0<p\le 1$. Suppose $\XB$ is nearly unconditional. There is a constant $C_2\in(0,\infty)$ such that, for all $0<a\le 1$ and all $0<b<1$,
\[
\phi(a)\le \frac{C_1\theta(b)}{b} \left(1+\frac{\log(a)}{\log(b)}\right)^{1/p}\rho(a).
\]
Thus there is $C_3\in(0,\infty)$ such that, for all $0<a\le 1$,
\begin{equation}\label{eq:anso11}
\phi(a)\le C_3(1-\log (a))^{1/p}\rho(a).
\end{equation}
\end{proposition}

\begin{proof}
Choose $n\in\NN$ so that $b^n\le a<b^{n-1}$, that is
\[
n-1<\log_b(a)=\frac{\log (a)}{\log(b)} \le n.
\]
Taking into account that
\[
1+(1-b)^p\theta^{\, p}(b)\le 2 \theta^{\, p}(b),
\]
combining Lemma~\ref{lem:logbounds,sumc} with inequality~\eqref{rhomorebounds} gives
\begin{multline*}
\phi(a)\le \phi(b^n)\le \frac{C_1}{b} n^{1/p}\rho\left(b^n\right) \le \frac{2^{1/p} C_1 \theta(b)}{b} n^{1/p} \rho\left(b^{n-1}\right) \\
\le \frac{2^{1/p} C_1 \theta(b)}{b} n^{1/p} \rho(a)
\le \frac{2^{1/p} C_1 \theta(b)}{b} \left( 1 +\log_b a\right)^{1/p} \rho(a).
\end{multline*}
Since, for a fixed $0<b<1$, $1 +\log_b a\approx 1-\log a$, we are done.
\end{proof}

In light of Theorem~\ref{theorem: QGLC=NU}, Theorem~\ref{corollary: QGLCphibound}, and Theorem~\ref{thm:miguelbound}, we wonder how a function $f$ must be so that there is a nearly unconditional basis with $\phi\approx f$. Specifically, it seems to be unknown whether there is a nearly unconditional basis whose unconditionality threshold function does not have a logarithmic growth. For the time being, we try to shed light onto this question by proving that, oddly enough, the optimality of inequality \eqref{eq:anso11} ensures that the unconditionality threshold function $\phi$ grows slowly.

\begin{proposition}\label{prop:phigrowsslowly}
Let $\XB$ be a nearly unconditional basis of a $p$-Banach space $\XX$, $0<p\le 1$.
\begin{enumerate}[label=(\roman*), leftmargin=*,widest=ii]
\item \label{proposition: logbounds, lessthanall}
If
\[
\liminf_{a\to 0^+}\frac{\phi(a)}{\left(-\log(a)\right)^{1/p}\rho(a)}>0,
\]
there are $D$, $d>0$ such that
\[
\phi(a)\le D(1-\log (a) )^{d}, \quad 0<a\le 1.
\]
\item \label{proposition: logbounds, alternative}If
\[
\ell:=\limsup_{a\to 0^+}\frac{\phi(a)}{\left(-\log(a)\right)^{1/p}\rho(a)}>0,
\]
then for every $d>0$ there is $C_d>0$ such that
\[
\phi(a)\le C_da^{-d}, \quad 0<a\le 1.
\]
\end{enumerate}
\end{proposition}

\begin{proof}
The proof of both \ref{proposition: logbounds, lessthanall} and \ref{proposition: logbounds, alternative} is based on the following claim.

\begin{claim}
Let $C_1$ be as in Lemma~\ref{lem:logbounds,sumc}.
Suppose that $0<a<1$, $C_2\in(0,\infty)$, $M>1$ and an integer $n\ge 2M-1$ are such that
\[
-\frac{2 C_1^p C_2^p}{a^p \log (a) M}\le 1-2^{-p} \quad \mbox{ and } \quad \left(- \log(a^n)\right)^{1/p}\rho(a^n)\le C_2 \phi(a^n).
\]
Then, if $\alpha=1-1/M$,
\[
\phi(a^n)
\le \frac{2 \alpha^{1/p} C_1}{a} (n-1)^{1/p} \rho(a^{\alpha (n-1)}).
\]
\end{claim}

To prove the claim we set and $k_n=\floor {\alpha(n-1)}$. We have
\begin{align*}
\phi^{\, p}(a^n) &\le \frac{C_1^p}{a^p}\left(\sum_{k=1}^{k_n}\rho^{\, p}(a^{k})+\sum_{k=k_n+1}^n\rho^{\, p}(a^{k})\right)\\
&\le \frac{C_1^p}{a^p}\left(k_n \rho^{\, p}(a^{k_n})+\left(n-k_n\right)\rho^{\,p}(a^n)\right)\\
&\le \frac{C_1^p}{a^p}\left(k_n \rho^{\, p}(a^{k_n})+\left(1+\alpha+\frac{n}{M}\right)\rho^{\,p}(a^n)\right)\\
&\le \frac{C_1^p}{a^p}\left(k_n \rho^{\, p}(a^{k_n})+\frac{2n}{M}\rho^{\,p}(a^n)\right)\\
&\le \frac{C_1^p}{a^p}\left(k_n \rho^{\, p}(a^{k_n})-\frac{2 C_2^p n}{M\log(a^n)} \phi^{\,p}(a^n)\right)\\
&\le \frac{C_1^p}{a^p} k_n \rho^{\, p}(a^{k_n}) +(1-2^{-p}) \phi^{\,p}(a^n).
\end{align*}
Hence,
\[
\phi(a^n)\le \frac{2 C_1}{a} k_n^{1/p} \rho\left(a^{k_n}\right)
\le \frac{2 \alpha^{1/p} C_1}{a} (n-1)^{1/p} \rho(a^{\alpha (n-1)}).
\]
as desired.

Now, we address proving \ref{proposition: logbounds, lessthanall}. Note that, by assumption, there is $0<C_2<\infty$ such that
\begin{equation*}
\left(-\log(a)\right)^{1/p}\rho(a)\le C_2 \phi(a)\quad 0<a\le 1.
\end{equation*}
Hence, we can apply the claim with such a constant $C_2$ and any $0<a<1$. To apply it with $a=e^{-1}$, we pick $M>1$ with
\[
M \ge \frac{2 e^p C_1^p C_2^p}{1-2^{-p} },
\]
and we set $\alpha=1-1/M$.
For every integer $n\ge 2M-1$ we have
\[
\phi(e^{-n})
\le 2 e \alpha^{1/p} C_1 (n-1)^{1/p} \rho(e^{-\alpha (n-1)})
\le 2 e C_1C_2 \phi(e^{-\alpha(n-1)}).
\]
We infer that there is a constant $K$ such that $\phi(e^{-n})\le K \phi(e^{-\alpha(n-1)})$ for all $n\in\NN$. Given $0<t\le 1$, pick $n\in\NN$ such that $e^{-n}<t\le e^{-(n-1)}$. We have
\[
\phi(t)\le \phi(e^{-n}) \le K \phi(e^{-\alpha(n-1)})\le K \phi(t^\alpha).
\]
Hence, the non-increasing map $\psi\colon(0,\infty) \to [1,\infty)$ given by $\psi(t)=\phi(e^{-1/t})$ satisfies
\[
\psi(\alpha t)\le K \psi (t), \quad t> 0.
\]
In particular, $\psi(\alpha^n)\le K \psi(\alpha^{n-1})$ for all $n\in\NN$. Set $D= K \phi(e^{-1})$ and $d=-\log_\alpha(K)$. By Lemma~\ref{lem:anso9},
\[
\phi(e^{-1/t})=\psi(t)\le D t^{-d}, \quad 0 < t\le 1.
\]
In other words,
\[
\phi(a) \le D (-\log a)^d, \quad 0<a\le e^{-1}.
\]
Since $\phi$ is bounded on $[e^{-1},1]$, the proof of \ref{proposition: logbounds, lessthanall} is complete.

Let $\Du$ be the set of indices $d\in(0,\infty)$ such that $\phi(a)\lesssim a^{-d}$ for $0<a\le 1$.
By Lemma~\ref{lem:fromsomeatoanya}, $ (\eta,\infty)\subset \Du$, where
\[
\eta:=\liminf_{c\to 0^+} -\frac{1}{p} \log_c \left(1+(1-c)^p\phi^{\, p}(c)\right)
=\liminf_{c\to 0^+} \frac {\log(\phi(c))}{-\log (c)}.
\]
So, in order to prove \ref{proposition: logbounds, alternative} it suffices to show that $\eta=0$. To that end, we will use a bootstrap argument. Namely, we will check that there is $\alpha\in(0,1)$ such that $d\in\Du$ implies $\eta\le \alpha d$. For that, we apply our claim. Pick $C_2>1/\ell$ and set $\alpha=1-1/M$, where $M>1$ satisfies
\[
M\ge \frac{2 e^{2p} C_1^p C_2^p}{1-2^{-p}}.
\]
Notice that
\[
-\frac{2 C_1^p C_2^p}{a^p \log (a) M}\le 1-2^{-p}, \quad e^{-2} \le a \le e^{-1}.
\]
Now, we pick a sequence $(c_m)_{m=1}^\infty$ in $(0,e^{ -2M}]$ with $\lim_m c_m=0$ and
\[
\left(- \log(c_m)\right)^{1/p}\rho(c_m)\le C_2\phi(c_m), \quad m\in\NN.
\]
Set, for each $m\in\NN$, $n_m=\floor{-\log(c_m)}$. Since $- \log(c_m)\le 1+n_m$, we have $n_m\ge 2M-1$. In particular, $n_m\ge 1$ and, hence,
\[
n_m \le - \log(c_m) \le 2 n_m.
\]
We infer that $c_m=a_m^{n_m}$ for some $a_m\in[e^{-2},e^{-1}]$. Therefore, applying the claim with $a=a_m$ gives
\[
\phi(c_m)
\le \frac{2 \alpha^{1/p} C_1}{a_m} (n_m-1)^{1/p} \rho(a_m^{\alpha (n_m-1)})
\le 2 e \alpha^{1/p} \left(-\log (c_m)\right)^{1/p} \rho(c_m^\alpha).
\]
Using Remark~\ref{rem:scale}, we infer that
\[
\eta\le \liminf_{c\to 0^+} \frac{ \log(\rho(c^\alpha))}{-\log(c)}
\le \liminf_{c\to 0^+} \frac{ \log(\phi(c^\alpha))}{-\log(c)}.
\]
Hence, if $d\in\Du$,
\[
\eta \le \liminf_{c\to 0^+} \frac{ \log(c^{-\alpha d})}{-\log(c)}=\alpha d.\qedhere
\]
\end{proof}

Roughly speaking, Proposition~\ref{prop:phigrowsslowly} says that if $\psi/\rho$ grows as fast as possible, then $\psi$ grows slowly. Along the same lines, we finish this section with a result that points out that if $\rho$ grows fast and steadily, then $\rho$ grows as $\phi$. We formulate it in terms of \emph{essentially decreasing} functions, i.e., functions $f\colon I \subset \RR\to\RR$ for which there is a constant $D\ge 1$ such that $f(b) \le D f(a)$ whenever $a\le b$.

\begin{proposition}\label{proposition: logbounds}
Let $\XB$ be a nearly unconditional basis of a $p$-Banach space $\XX$, $0<p\le 1$.
\begin{enumerate}[label=(\roman*), leftmargin=*,widest=ii]
\item\label{proposition: logbounds, phi=rho}
Suppose there are $0<c<1$ and $C>0$ such that
\[
\left(\sum_{k=1}^n\rho^{\, p}(c^k)\right)^{1/p}\le C \rho(c^{n+1}), \quad n\in\NN.
\]
Then $\rho\approx \lambda \approx \theta \approx \phi$. In particular, this holds if there is $d>0$ such that the function $a\mapsto a^d \rho(a)$, $0<a\le 1$, is essentially decreasing.
\item\label{proposition: logbounds, phi=theta}
Suppose there are $0<c<1$ and $C>0$ such that
\[
\left(\sum_{k=1}^n\theta^{\, p}(c^k)\right)^{1/p}\le C \theta(c^{n+1}), \quad n\in\NN.
\]
Then $\theta \approx \phi$. In particular, this holds if there is $d>0$ such that the function $a\mapsto a^d \theta(a)$, $0<a\le 1$, is essentially decreasing.
\end{enumerate}
\end{proposition}

\begin{proof}
By Remark~\ref{rem:scale}, in order to show \ref{proposition: logbounds, phi=rho} it suffices to prove that $\phi\lesssim\rho$.
Now, to carry out a unified proof of \ref{proposition: logbounds, phi=rho} and \ref{proposition: logbounds, phi=theta}, we set $\psi=\rho$ in the former case, and $\psi=\theta$ in the latter.

By Remark~\ref{rem:scale} and inequality~\eqref{rhomorebounds}, there is a constant $K$ (depending on $c$) such that
\[
\rho(c^{n+1})\le K \psi(c^{n-1}), \quad n\in\NN.
\]
Given $a\in(0,1]$, choose $n\in\NN$ so that $c^n< a\le c^{n-1}$. Let $C_1$ be the constant provided by Lemma~\ref{lem:logbounds,sumc}. We have
\begin{align*}
\phi^{\, p}(a)
&\le \phi(c^n)\\
&\le \frac{C_1^p}{c^p} \left( \rho^{\, p}(c)+\rho^{\, p}(c^2) +K^p \sum_{k=2}^n \psi^{\, p} (c^{k-2})\right)\\
&\le \frac{C_1^p}{c^p} \left( \rho^{\, p}(c)+\rho^{\, p}(c^2) +K^p C^p \psi^{\, p}(c^{n-1})\right)\\
&\le \frac{C_1^p}{c^p} \left( \rho^{\, p}(c)+\rho^{\, p}(c^2) +K^p C^p \psi^{\, p}(a)\right).
\end{align*}
Since $\psi$ is bounded away from zero, $\phi\lesssim\psi$, as desired.

If the function $a\mapsto a^d \psi(a)$ is essentially decreasing for some $d>0$, then there is $D\ge 1$ such that
\[
\psi(b)\le D \left(\frac{a}{b}\right)^d \psi(a), \quad 0<a\le b \le 1.
\]
Hence, for any $c\in(0,1)$ and $n\in\NN$,
\[
\sum_{k=1}^n\psi^{\, p}(c^k) \le D^p \psi(c^{n+1}) \sum_{k=1}^n c^{(n+1-k)dp}\le\frac{D^p c^{dp}}{1-c^{dp}} \psi(c^{n+1}).
\]
So, the condition of the statement holds.
\end{proof}

\section{Isometric QGLC bases}\label{sect:iso}\noindent
In non-linear approximation theory in Banach spaces, the special case of greedy-like bases with constant $1$ is of special interest, and has been studied throughout several papers. The most relevant results within this area are, probably, the characterizations of $1$-greedy, $1$-almost greedy, and $1$-quasi-greedy bases.
\begin{itemize}
\item A basis of a Banach space is $1$-greedy if and only if it is $1$-suppression unconditional and $1$-symmetric for largest coefficients (see \cite{AW2006}*{Theorem 3.4}).
\item A basis of a Banach space is $1$-almost greedy if and only if it is $1$-symmetric for largest coefficients (see \cite{AlbiacAnsorena2017b}*{Theorem 1.5}).
\item A basis of a Banach space is $1$-quasi greedy if and only if it is $1$-suppression unconditional (see \cite{AlbiacAnsorena2016c}*{Theorem 2.1}).
\end{itemize}

In this context, it is natural to wonder whether there is also a characterization of $1$-QGLC bases in terms of other properties that have already appeared in the literature. The main result of this section addresses this question. Before stating it, we given an auxiliary lemma.

\begin{lemma}\label{lem:anso12}
Let $(\xx_n)_{n\in A}$ be a finite family in a Banach space $\XX$. Suppose that
\[
\norm{\sum_{n\in B} \xx_n} \le \norm{\sum_{n\in A} \xx_n}
\]
for all $B\subset A$. Then, for all $(\lambda_n)_{n\in A}$ in $[1,\infty)$,
\[
\norm{\sum_{n\in A} \xx_n} \le \norm{\sum_{n\in A}\lambda_n \, \xx_n}.
\]
\end{lemma}
\begin{proof}
Use the Hahn-Banach theorem to pick $f^*\in B_{\XX^*}$ such that
\[
\norm{\sum_{n\in A} \xx_n}=f^*\left( \sum_{n\in A} \xx_n \right).
\]
Set $B=\{n\in A \colon \Re(f^*(\xx_n))\ge 0\}$. We have
\[
\sum_{n\in B} \Re(f^*(\xx_n)) \le \abs {f^*\left( \sum_{n\in B} \xx_n\right)} \le \norm{\sum_{n\in B} \xx_n}
=\sum_{n\in A} \Re(f^*(\xx_n)).
\]
We infer that $A\setminus B=\emptyset$, that is, $\Re(f^*(\xx_n))\ge 0$ for all $n\in A$. Hence, given $(\lambda_n)_{n\in A}$ in $[1,\infty)$,
\begin{align*}
\norm{\sum_{n\in A} \xx_n}=\sum_{n\in A} \Re(f^*(\xx_n))&\le \sum_{n\in A}\lambda_n \Re(f^*(\xx_n))\\
&\le \abs {f^*\left( \sum_{n\in A} \lambda_n\, \xx_n\right)}\le \norm{\sum_{n\in A}\lambda_n \, \xx_n}.\qedhere
\end{align*}
\end{proof}

\begin{proposition}\label{proposition: 1-QGLC}Let $\XB=(\xx_n^*)_{n=1}^\infty$ be a basis of a Banach space $\XX$. The following are equivalent:
\begin{enumerate}[label=(\roman*), leftmargin=*,widest=iv]
\item \label{proposition: 1-QGLC, 1-TQG}$\XB$ is $1$-truncation quasi-greedy.
\item \label{proposition: 1-QGLC, lambda} $\XB$ is nearly truncation quasi-greedy with $\lambda(a)=1$ for all $0<a\le 1$.
\item \label{proposition: 1-QGLC, rho} $\XB$ is nearly truncation quasi-greedy with $\rho(a)=1$ for all $0<a\le 1$.
\item \label{proposition: 1-QGLC, 1-QGLC}$\XB$ is $1$-quasi-greedy for largest coefficients.
\item \label{proposition: 1-QGLC, 1-other}For every finite set $A\subset \NN$, every $\varepsilon\in\EE^A$, and every $f\in\XX$ with $\supp(f)\cap A=\emptyset$,
\[
\norm{\Ind_{\varepsilon,A}}\le \norm{ \Ind_{\varepsilon,A}+f }.
\]
\end{enumerate}
\end{proposition}
\begin{proof}
\ref{proposition: 1-QGLC, 1-TQG}$\Longrightarrow$\ref{proposition: 1-QGLC, lambda} and \ref{proposition: 1-QGLC, lambda}$\Longrightarrow$ \ref{proposition: 1-QGLC, rho} are immediate from the definitions, and \ref{proposition: 1-QGLC, rho} $\Longrightarrow$ \ref{proposition: 1-QGLC, 1-QGLC} follows from Lemma~\ref{remark: similar}.

\noindent \ref{proposition: 1-QGLC, 1-QGLC}$\Longrightarrow$ \ref{proposition: 1-QGLC, 1-other} Assume, by contradiction, that there are $A\subset \NN$, $\varepsilon\in\EE^A$, and $f\in\XX$ with $\supp(f)\cap A=\emptyset$ with $\norm{\Ind_{\varepsilon,A}+f} <\norm{\Ind_{\varepsilon,A}}$. The mapping $F\colon\RR\to \RR$ given by
\[
F(t)=\norm{\Ind_{\varepsilon,A}+tf}, \quad t\in\RR,
\]
is convex, and we have $F(0)=\norm{\Ind_{\varepsilon,A}}$ and $F(1)<\norm{\Ind_{\varepsilon,A}}$. Hence, $F(t)<\norm{\Ind_{\varepsilon,A}}$ for every $0<t<1$. Choosing $t\in(0,1)$ small enough we have $tf\in\Cu$, so that, by assumption, $\norm{\Ind_{\varepsilon,A}}\le F(t)$. We have reached an absurdity, as desired.

\noindent \ref{proposition: 1-QGLC, 1-other} $\Longrightarrow$ \ref{proposition: 1-QGLC, 1-TQG}
Fix a finite set $A\subset\NN$. Let $\YY$ be the closed linear subspace of $\XX$ given by
\[
\YY=\{ f\in\XX \colon \supp(f) \cap A =\emptyset\},
\]
and let $Q\colon \XX \to \XX/\YY$ be the canonical quotient map. Fix now $\varepsilon=(\varepsilon_n)_{n\in A}\in\EE^A$, and set
$\zz_n =Q(\varepsilon_n\, \xx_n)$ for all $n\in A$. By assumption,
\[
\norm{\sum_{n\in A} \zz_n}=\norm {Q(\Ind_{\varepsilon,A})}=\norm{\Ind_{\varepsilon,A}}.
\]
Moreover, for each $B\subset A$ we have
\[
\norm{\sum_{n\in B} \zz_n}
=\norm {Q(\Ind_{\varepsilon,B} )}
\le \norm{\Ind_{\varepsilon,B}}
\le \norm{\Ind_{\varepsilon,A}}.
\]
Hence, we can apply Lemma~\ref{lem:anso12} to the family $(\zz_n)_{n\in A}$. Thus we obtain
\[
\norm{\Ind_{\varepsilon,A} } \le \norm{f+\sum_{n\in A}\lambda_n\,\varepsilon_n \,\xx_n}
\]
for all $(\lambda_n)_{n\in A}$ in $[1,\infty)$ and all $f\in\XX$ with $\supp(f)\cap A=\emptyset$. This means that $\XB$ is $1$-truncation quasi-greedy.
\end{proof}

\section{Open questions}\label{sect:questions}\noindent
To the best of our knowledge, it is unknown whether there are QGLC bases that are not truncation quasi-greedy. In light of Theorem~\ref{theorem: QGLC=NU}, this question extends beyond the bounds of approximation theory to become central within the theory of bases.

\begin{question}\label{qtn:A}
Is there a nearly unconditional basis that is not truncation quasi-greedy? If the answer were positive, the problem of finding conditions on the quasi-Banach space $\XX$ which ensure that all nearly unconditional bases of $\XX$ are truncation quasi-greedy would make sense.
\end{question}

Since, by Proposition~\ref{proposition: 1-QGLC}, we know that $1$-QGLC implies truncation quasi-greediness in the framework of Banach spaces, Question~\ref{qtn:A} connects with the problem of finding renormings that improve the QGLC constant of the basis. Besides, by Theorem~\ref{corollary: QGLCphibound}, such a renorming would lead to a better control of the threshold unconditionality function $\phi$. So, the renorming problem is of particular interest in this context.

\begin{question}\label{question: renorming}
Let $\XX$ be a Banach space. Is there a constant $C$ such that any nearly unconditional basis of $\XX$ becomes $C$-QGLC under a suitable renorming of $\XX$? Does this hold with $C=1$ or, at least, $C=1+\epsilon$ for any $\epsilon>0$?
\end{question}

Note that an affirmative answer to Question~\ref{question: renorming} would give an absolute bound for the growth of the threshold unconditionality function.


\begin{bibdiv}
\begin{biblist}

\bib{AlbiacAnsorena2016c}{article}{
author={Albiac, F.},
author={Ansorena, J.~L.},
title={Characterization of 1-quasi-greedy bases},
date={2016},
ISSN={0021-9045},
journal={J. Approx. Theory},
volume={201},
pages={7\ndash 12},
url={https://doi.org/10.1016/j.jat.2015.08.006},
review={\MR{3424006}},
}

\bib{AlbiacAnsorena2017b}{article}{
author={Albiac, Fernando},
author={Ansorena, Jos\'{e}~L.},
title={Characterization of 1-almost greedy bases},
date={2017},
ISSN={1139-1138},
journal={Rev. Mat. Complut.},
volume={30},
number={1},
pages={13\ndash 24},
url={https://doi-org/10.1007/s13163-016-0204-3},
review={\MR{3596024}},
}

\bib{AAB2022}{article}{
author={Albiac, Fernando},
author={Ansorena, Jos\'{e}~L.},
author={Berasategui, Miguel},
title={Sparse approximation using new greedy-like bases in
superreflexive spaces},
date={2022},
journal={arXiv e-prints},
eprint={2205.09478},
url={https://arxiv.org/abs/2205.09478},
}

\bib{AABBL2022}{article}{
author={Albiac, Fernando},
author={Ansorena, Jos\'{e}~L.},
author={Berasategui, Miguel},
author={Bern\'{a}, Pablo~M.},
author={Lassalle, Silvia},
title={Weak forms of unconditionality of bases in greedy approximation},
date={2022},
ISSN={0039-3223},
journal={Studia Math.},
volume={267},
number={1},
pages={1\ndash 17},
url={https://doi-org/10.4064/sm210601-2-2},
review={\MR{4460220}},
}

\bib{AABW2021}{article}{
author={Albiac, Fernando},
author={Ansorena, Jos\'{e}~L.},
author={Bern\'{a}, Pablo~M.},
author={Wojtaszczyk, Przemys{\l}aw},
title={Greedy approximation for biorthogonal systems in quasi-{B}anach
spaces},
date={2021},
journal={Dissertationes Math. (Rozprawy Mat.)},
volume={560},
pages={1\ndash 88},
}

\bib{AW2006}{article}{
author={Albiac, Fernando},
author={Wojtaszczyk, Przemys{\l}aw},
title={Characterization of 1-greedy bases},
date={2006},
ISSN={0021-9045},
journal={J. Approx. Theory},
volume={138},
number={1},
pages={65\ndash 86},
url={https://doi.org/10.1016/j.jat.2005.09.017},
review={\MR{2197603}},
}

\bib{BB2022b}{article}{
author={Berasategui, M.},
author={Bern\'{a}, P.~M.},
title={Extensions of greedy-like bases for sequences with gaps},
date={2022},
url={https://arxiv.org/abs/2009.02257},
}

\bib{BBG2017}{article}{
author={Bern\'{a}, Pablo~M.},
author={Blasco, \'{O}scar},
author={Garrig\'{o}s, Gustavo},
title={{L}ebesgue inequalities for the greedy algorithm in general
bases},
date={2017},
ISSN={1139-1138},
journal={Rev. Mat. Complut.},
volume={30},
number={2},
pages={369\ndash 392},
url={https://doi.org/10.1007/s13163-017-0221-x},
review={\MR{3642039}},
}

\bib{DOSZ2009}{article}{
author={Dilworth, S.~J.},
author={Odell, E.},
author={Schlumprecht, Th.},
author={Zs\'{a}k, A.},
title={Partial unconditionality},
date={2009},
ISSN={0362-1588},
journal={Houston J. Math.},
volume={35},
number={4},
pages={1251\ndash 1311},
review={\MR{2577152}},
}

\bib{DKK2003}{article}{
author={Dilworth, Stephen~J.},
author={Kalton, Nigel~J.},
author={Kutzarova, Denka},
title={On the existence of almost greedy bases in {B}anach spaces},
date={2003},
ISSN={0039-3223},
journal={Studia Math.},
volume={159},
number={1},
pages={67\ndash 101},
url={https://doi.org/10.4064/sm159-1-4},
note={Dedicated to Professor Aleksander Pe{\l}czy\'nski on the occasion
of his 70th birthday},
review={\MR{2030904}},
}

\bib{DKKT2003}{article}{
author={Dilworth, Stephen~J.},
author={Kalton, Nigel~J.},
author={Kutzarova, Denka},
author={Temlyakov, Vladimir~N.},
title={The thresholding greedy algorithm, greedy bases, and duality},
date={2003},
ISSN={0176-4276},
journal={Constr. Approx.},
volume={19},
number={4},
pages={575\ndash 597},
url={https://doi-org/10.1007/s00365-002-0525-y},
review={\MR{1998906}},
}

\bib{Elton1978}{book}{
author={Elton, John~Hancock},
title={Weakly null normalized sequences in {B}anach spaces},
publisher={ProQuest LLC, Ann Arbor, MI},
date={1978},
url={http://gateway.proquest.com/openurl?url_ver=Z39.88-2004&rft_val_fmt=info:ofi/fmt:kev:mtx:dissertation&res_dat=xri:pqdiss&rft_dat=xri:pqdiss:7915816},
note={Thesis (Ph.D.)--Yale University},
review={\MR{2628434}},
}

\bib{KoTe1999}{article}{
author={Konyagin, Sergei~V.},
author={Temlyakov, Vladimir~N.},
title={A remark on greedy approximation in {B}anach spaces},
date={1999},
ISSN={1310-6236},
journal={East J. Approx.},
volume={5},
number={3},
pages={365\ndash 379},
review={\MR{1716087}},
}

\bib{Rosenthal1974}{article}{
author={Rosenthal, Haskell~P.},
title={A characterization of {B}anach spaces containing {$l^{1}$}},
date={1974},
ISSN={0027-8424},
journal={Proc. Nat. Acad. Sci. U.S.A.},
volume={71},
pages={2411\ndash 2413},
review={\MR{358307}},
}

\end{biblist}
\end{bibdiv}
\end{document}